\documentclass[12pt]{amsart}

\usepackage{amsmath,amssymb,amsmath,epic,lscape}

\newtheorem{theorem}{Theorem}[section]
\newtheorem{proposition}[theorem]{Proposition}

\theoremstyle{definition}

\theoremstyle{remark}

\numberwithin{equation}{section}

\def\DJ{{\hbox{D\kern-.8em\raise.15ex\hbox{--}\kern.35em}}}
\def\DJo{$\;$\kern-.4em
    \hbox{D\kern-.8em\raise.15ex\hbox{--}\kern.35em okovi\'c}}

\renewcommand{\subjclassname}{\textup{2000} Mathematics Subject
Classification }

\begin{document}

\title[Small orders of Hadamard matrices]
{Small orders of Hadamard matrices and base sequences}

\author[D.\v{Z}. \DJo]
{Dragomir \v{Z}. \DJo}
\address{Department of Pure Mathematics and Institute for Quantum Computing, University of Waterloo,
Waterloo, Ontario, N2L 3G1, Canada}
\email{djokovic@uwaterloo.ca}

\keywords{Hadamard matrices, base sequences, T-sequences,
orthogonal designs, Williamson-type matrices}

\date{}

\begin{abstract}
We update the list of odd integers $n<10000$ for which an Hadamard 
matrix of order $4n$ is known to exist. We also exhibit the first 
example of base sequences $BS(40,39)$. Consequently, there
exist T-sequences $TS(n)$ of length $n=79$. The first undecided 
case has the length $n=97$.
\end{abstract}

\maketitle
\subjclassname{ 05B20, 05B30 }
\vskip5mm

\section{Introduction}

We ask a very simple question: For which odd positive integers
$n<10000$ is it known how to construct an Hadamard matrix of
order $4n$? We shall refer to such $n$ (in this range) as 
{\em good} integers, and to other as ``bad''. 
Unfortunately, in spite of the fact that the Hadamard matrix 
conjecture is very old and constitutes a very active area of
current research in combinatorics, the answer to this question
is apparently not known. As a tentative answer we choose 
\cite[Table 1.53]{HCD}. In fact this table
is more ambitious as it also provides the least exponent $t$
for which it is known that an Hadamard matrix of order $2^t n$
exists. We have qualified this answer as tentative for
two reasons. First of all the table has been published three
years ago and needs to be updated. For instance we have 
constructed Hadamard matrices of order 764 (see \cite{DZ5}).
Secondly, the information contained in the table was not accurate
even at the time of its publication. Indeed, in our note
\cite{DZ6}, we have given a list of 138 good values of $n$, 
which have not been recorded in the table. In the same note, 
by using new results, we have shown that 4 additional values 
of $n$ are good.

We can now replace the above question with two simpler
questions. First, is it true that the integers $n$ asserted to 
be good in \cite[Table 1.53]{HCD} are indeed good? They probably 
are (and we continue to consider them as good) but I admit that
I was not able to verify this assertion in all cases since the
references provided are not sufficient and the 
literature on this subject is enormous. The tables in the old 
survey paper \cite{SY2} are much better in that regard as they 
include the information necessary for the construction of 
tabulated matrices. The second simplified question is:
can we convert some of the bad integers into good ones? We shall 
address only the latter question in this note.

We can state our main result simply by saying that we have
converted 42 bad integers into good ones. (Most of them were
good even three years ago.) Originally, i.e., according to
\cite[Table 1.53]{HCD} there were 1006 bad integers. The update 
in \cite{DZ6} reduced this number to 864, and here
we reduce it further to 822.

We refer the reader to \cite{HCD,SY2} and to our note 
\cite{DZ6} for the standard definitions and notation. 
As in that note, we shall write $OD(4d)$ for the orthogonal 
design $OD(4d;d,d,d,d)$. It is now known that T-sequences 
of length $d$ exist, and consequently $OD(4d)$ exist, for 
all $d\le100$ except possibly for $d=97$. For this see the 
next section where we recall some old results and present 
a new one. We use these results later to construct some 
particular Hadamard matrices that we need.

\section{Tools for the construction}

Our objective is to show how one can construct Hadamard 
matrices of order $4n$ for the following 42 odd integers
$n$:

\begin{eqnarray*}
&& 787, 823, 883,1063,1303,1527,2143,2335,2545,2571, \\
&& 3533,4285,5441,5449,5999,6181,6617,6819,7167,7179, \\
&& 7251,7323,7663,7779,8067,8079,8139,8187,8237,8259, \\
&& 8499,8573,8611,8653,8751,8859,9111,9123,9427,9627, \\
&& 9671,9939.
\end{eqnarray*}
(According to \cite[Table 1.53]{HCD} they are all bad.)
This list does not overlap with the list of 142 good 
numbers in \cite{DZ6}.

The construction is based on the following old results and
on a new result that we will mention afterwards.

First, we need Mathon's theorem about symmetric conference 
matrices. Recall that a square matrix $C$ of order $n$ is 
called a {\em conference matrix} if its diagonal entries are 0,
its off-diagonal entries are $\pm1$, and $CC^T=(n-1)I_n$,
where $T$ denotes transposition and $I_n$ is the identity matrix. 
If a conference matrix is symmetric, its order $n$ must be 1 
or $\equiv 2 \pmod{4}$.

\begin{theorem} \rm{( Mathon \cite{RM} )} If $q\equiv 3\pmod{4}$
is a prime power and $q+2$ is a prime power, then there exists
a symmetric conference matrix of order $q^2(q+2)+1$ and a 
symmetric Hadamard matrix of order $2q^2(q+2)+2$.
\end{theorem}

Next, we need three theorems of Yamada which we compress into two.

\begin{theorem} \rm{( Yamada \cite{MY} )} Let $q\equiv 1\pmod{8}$
be a prime power. 

(a) If there exists an Hadamard matrix of order $(q-1)/2$, then
there exists an Hadamard matrix of order $4q$.

(b) If there exists a symmetric conference matrix of order 
$(q+3)/2$, then there exists an Hadamard matrix of order 
$4(q+2)$.
\end{theorem}

Let us also recall that a {\em skew Hadamard matrix} is a
Hadamard matrix $H$ (of order $n$ say) such that $H-I_n$ is a 
skew-symmetric matrix.

\begin{theorem} \rm{( Yamada \cite{MY} )} If $q\equiv 5\pmod{8}$
is a prime power and there exists a skew Hadamard matrix of order 
$(q+3)/2$, then there exists an Hadamard matrix of order $4(q+2)$.
\end{theorem}

(For this theorem and part (b) of the previous one, Yamada gives 
credit to Z. Kiyasu.)

We also need two results of Miyamoto. The first one is the
following theorem.

\begin{theorem} \rm{( Miyamoto \cite{MM} )} If $q\equiv 1\pmod{4}$
is a prime power and there exists an Hadamard matrix of order $q-1$, 
then there exists an Hadamard matrix of order $4q$.
\end{theorem}

For the second we need to recall the definition of 
Williamson-type matrices. Two matrices $A,B$ of order $n$ are 
{\em amicable} if $AB^T=BA^T$. Four $\{+1,-1\}$-matrices 
$A,B,C,D$ of order $n$ are called {\em Williamson-type matrices} 
if they are pairwise amicable and satisfy
$$ AA^T+BB^T+CC^T+DD^T=4nI_n. $$

We quote the second result of Miyamoto from the presentation
provided by Seberry and Yamada \cite[Corollary 8.8, part 1]{SY2}
or \cite[Corollary 29, part (i)]{SY1} where the proof is 
also given.

\begin{theorem}  
Let $q\equiv 1\pmod{4}$ be a prime power. Then there exist 
Williamson-type matrices of order $q$ if there are 
Williamson-type matrices of order $(q-1)/4$ or an Hadamard 
matrix of order $(q-1)/2$.
\end{theorem}

The new result that we need is obtained by means of a computer.
Namely, we have constructed the first example of base
sequences $BS(40,39)$. A well-known construction then gives us
T-sequences of length 79, and also the orthogonal 
design $OD(4d)$ with $d=79$. This result will be used in the
proof of Proposition \ref{Stav-4} in the next section.
We recall that we have constructed in \cite{DZ8} T-sequences 
of length 73. Hence (see e.g., \cite[Remark 8.47]{HCD}) 
T-sequences $TS(n)$ of length $n\le100$ all exist except 
possibly for $n=97$.

The base sequences $(A;B;C;D)\in BS(40,39)$ that we have found 
are given in encoded form by
$$ [06582 25343 22457 18723, 11644 42638 46572 23422]. $$
The encoding scheme is explained in \cite{DZ7}.
For the reader's convenience we also give these sequences
explicitly by writing $+$ for $+1$ and $-$ for $-1$:

\begin{eqnarray*}
A &=& ++--+ +--+- +++-- +--+- \\
&& ++--+ -+-++ +-+++ +-+-+; \\
B &=& +++-- -+--- ---++ +-+-- \\
&& +-+-+ +-+-- +++-- -----; \\
C &=& +++++ +++-- ++--+ +-++ - \\
&& +-++ +-+-- -+-+- ---++; \\
D &=& +++-- --+-- -+++- ---- + \\
&& -++- -+--+ -+--+ ++-++. \\
\end{eqnarray*}

The Base Sequence Conjecture (BSC) asserts that all 
$BS(n+1,n)$ exist, i.e., are nonvoid. Due to the above example, 
we can now update its status (see \cite{DZ7}): BSC  has been 
verified for all $n\le40$ and is also known to be valid for all 
Golay numbers $n=2^a10^b26^c$ ($a,b,c$ nonnegative integers).

\section{Existence of some Hadamard matrices}

For convenience, we split the proof into four propositions.
We consider first the prime integers $n$.

\begin{proposition} \label{Stav-1}
For each of the primes
$$ n=787,823,883,1063,1303,2143,3533,5441,5449,8237,8573 $$
there exists an Hadamard matrix of order $4n$.
\end {proposition}

\begin{proof}
The case $n=787$ is an instance of Mathon's theorem.
Indeed for $q=11$ we have $q^2(q+2)+1=1574=2\cdot 787$.

In the cases $n=5441,5449$ we apply part (a) of the first 
Yamada theorem with $q=n$. The existence of required Hadamard
matrices of order $(q-1)/2$ has been known for long time
(see e.g. \cite{SY2}).

In the case $n=883$ we use part (b) of the first Yamada theorem
with $q=n-2=881$ (a prime). Then $(q+3)/2=442$ and, by 
Mathon's theorem, there exists a symmetric conference matrix 
of order $442$.

In the cases $n=823,1063,1303,2143$ we apply the second Yamada
theorem with $q=n-2$. The required skew Hadamard matrices 
of order $(q+3)/2=4\cdot103,4\cdot133,4\cdot163,2^4\cdot67$
exist (see \cite{DZ1,DZ2,DZ4}).

In the remaining three cases $n=3533,8237,8573$ these primes 
are $\equiv 5\pmod{8}$ and we can apply the first Miyamoto 
theorem with $q=n$. This theorem requires $q-1=4d$ to be the 
order of an Hadamard matrix. If $q=3533,8573$ we have 
$d=883,2143$. Both of these cases have been already handled 
in the previous paragraphs. In the remaining case $8237$ 
we have $d=2059$. We can easily handle this case since 
$2059=29\cdot71$ and we know that there exist Williamson 
matrices of order 29 as well as T-sequences of length 71 
(e.g., see \cite{SY2} and \cite{KS} or \cite{DZ7}). This 
implies the existence of an Hadamard matrix of order 
$4d$ (see \cite{DZ6}).
\end{proof}

(The first case, $n=787$, could have been identified as 
good not only in \cite[Table 1.53]{HCD} but also in \cite{SY2}.)

\begin{proposition} \label{Stav-2}
For each of the numbers
\begin{eqnarray*}
n &=& 2571,6819,7179,7251,7323,7779,8067,8139, \\
&& 8187,8259,8499,8859,9087,9123,9939
\end{eqnarray*}
there exists an Hadamard matrix of order $4n$.
\end {proposition}

\begin{proof} 
Note that in all cases we have $n=3q$ where
\begin{eqnarray*}
q &=& 857,2273,2393,2417,2441,2593,2689,2713,2729,2753, \\
&& 2833,2953,3041,3209,3313
\end{eqnarray*}
is a prime $\equiv 1\pmod{4}$. Since the orthogonal design 
$OD(4d)$ exists for $d=3$, it suffices to show that, for each 
of the 15 values of $q$, there exist Williamson-type matrices 
of order $q$. This can be deduced from the second Miyamoto 
theorem. Indeed, it suffices to verify that there exists an 
Hadamard matrix of order $(q-1)/2=4m$ where
$$ m=107,284,299,302,305,324,336,339,341,344,354,369,
380,401,414. $$
For $m=107$ see \cite{KT} and for all other see \cite{SY2}.
\end{proof}

It is clear from this proof that all of these cases but the 
first could have been recorded in \cite[Table 1.53]{HCD}.

\begin{proposition} \label{Stav-3}
For each $n=1527,7167,8079,8751,9111$ there exists an Hadamard 
matrix of order $4n$.
\end {proposition}
\begin{proof}
We have $n=3q$ where $q=509,2389,2693,2917,3037$.
Again we use the $OD(4d)$ for $d=3$ and our task is to show 
that there exist Williamson-type matrices of order $q$. 
As each of these $q$ is a prime $\equiv 1\pmod{4}$, 
we can use again the second Miyamoto theorem. 
This time we verify that there exist Williamson-type matrices 
of order $(q-1)/4=127,597,673,729,759.$
For $127$ see \cite{DZ3} and for all other see \cite{SY2}.
\end{proof}

\begin{proposition} \label{Stav-4}
For each of the numbers
$$n=2335,2545,4285,5999,6181,6617,7663,8611,8653,9427,9671$$
there exists an Hadamard matrix of order $4n$.
\end {proposition}
\begin{proof}
For $n=2335$ we shall apply the second Yamada theorem with
$q=2333$, a prime $\equiv 5 \pmod{8}$. We have to verify that 
there exists a skew Hadamard matrix of order 
$(q+3)/2=2^4\cdot73$. This is indeed true because there is 
an infinite series of skew Hadamard matrices constructed by
E. Spence \cite{ES} which contains such a matrix of order 
$4\cdot73$.

Each of the remaining numbers is a product of two distinct 
primes, say $n=dm$ with $d<m$ and $m=97,109,509,857,883$. 
(There are only 5 different $m$.)
Since in all cases $d<97$, we know that
T-sequences of length $d$ exist, and consequently also
the orthogonal design $OD(4d)$ exists. It remains to show
that there exist Williamson-type matrices of order $m$.
For $m=97$ and $m=109$ see \cite{SY2} and for $m=509,857,883$ 
see the proofs of Propositions \ref{Stav-3},\ref{Stav-2} and 
\ref{Stav-1}, respectively.
\end{proof}

\section{Acknowledgments}

The author is grateful to NSERC for the continuing support of
his research. Part of this work was made possible by the facilities 
of the Shared Hierarchical Academic Research Computing Network 
(SHARCNET:www.sharcnet.ca).

\end{document}